\documentclass[oneside,reqno,12pt]{amsart}
\usepackage[greek,english]{babel}
\usepackage{fullpage}
\usepackage{amsthm}
\usepackage{amsbsy}
\usepackage{amsfonts}
\usepackage{graphicx}
\usepackage{hyperref}
\hypersetup{colorlinks=true, citecolor=blue, linkcolor=red}

\newtheorem{thm}{Theorem}

\newtheorem{pro}{Proposition}
\newtheorem{rem}{Remark}

\numberwithin{equation}{section} \numberwithin{lem}{section}
\numberwithin{thm}{section} \numberwithin{cor}{section}
\numberwithin{pro}{section} \numberwithin{rem}{section}

\begin{document}
\title[On the weak separation limit of a two-component BEC]{On the weak separation limit of a two-component Bose-Einstein condensate}
\author{Christos Sourdis}
\address{Department of Mathematics,  University of
Turin,   Via Carlo Alberto 10,
20123, Turin, Italy.}
\email{christos.sourdis@unito.it}

\date{\today}
\begin{abstract}
This paper deals with the study of the behaviour of the wave functions of a two-component
 Bose-Einstein condensate
  in the case of weak segregation. This
  amounts to the study of the asymptotic behaviour of a heteroclinic connection in a  conservative Hamiltonian system of two coupled second order ODE's,  as the strength of the coupling tends to its infimum.
  %, of heteroclinic solutions (''domain walls").
  % that connect two equilibria.
For this purpose, we apply geometric singular perturbation theory.
\end{abstract}
 \maketitle

%\tableofcontents

\section{Introduction}\label{secIntro}%\input{introaacs.tex}
\subsection{The problem}
We consider the following \emph{heteroclinic connection problem}:
\begin{equation}\label{eqEqGen}
	\left\{
	\begin{array}{rcl}
	\lambda^2 \ddot{u}&=&u^3-u+\Lambda v^2 u, \\
&&	\\
	  \ddot{v}&=&v^3-v+\Lambda u^2 v;
	\end{array}
	\right.
	\end{equation}
\begin{equation}\label{equvPos}
u,v>0;
\end{equation}
	\begin{equation}\label{eqBdryGen}
	(u,v )\to  (0,1)\ \textrm{as}\ z\to -\infty,\ \ (u,v)\to
	(1,0)\ \textrm{as}\ z\to  +\infty,
	\end{equation}
 for values of the parameter \[\Lambda>1,\] where for the constant $\lambda$ we may assume  without loss of generality  that $\lambda \geq 1$.

This problem arises in the study of two-component Bose-Einstein condensates in the case of segregation, see \cite{aftalionSourdis} and the references therein, but also in the study of certain amplitude equations (see \cite{pismen2006patterns,van2000domain}).

 The heteroclinic connection problem (\ref{eqEqGen})-(\ref{equvPos})-(\ref{eqBdryGen}) always admits a solution
which minimizes the associated enegy in Proposition \ref{proEnergy} below (see \cite{alamaARMA15} and \cite{van2000domain}). This type of heteroclinics enjoy the following monotonicity property:
\begin{equation}\label{eqmonotS}
\dot{u}>0,\  \dot{v}<0,
\end{equation}
(actually this is an implication of their stability, see \cite{alamaARMA15}); in the special case where $\lambda=1$, it also holds that the function \begin{equation}\label{eqARCT}\arctan(v/u)\ \ \textrm{is\ decreasing}\end{equation}  and $u(z+z_0)\equiv v(z_0-z)$ for some $z_0 \in \mathbb{R}$ (see \cite{van2000domain}). Moreover, any solution of (\ref{eqEqGen}), (\ref{eqBdryGen}) satisfies $u^2+v^2<1$ (see \cite{alamaARMA15}) and the hamiltonian identity
\begin{equation}\label{eqHam}
\lambda^2\frac{(\dot{u})^2}{2}+\frac{(\dot{v})^2}{2}-\frac{(1-u^2-v^2)^2}{4}-\frac{\Lambda-1}{2}u^2v^2\equiv 0.
\end{equation}
Remarkably, if there were more general coefficients in (\ref{eqEqGen}), then they could be absorbed in $\lambda, \Lambda$ by a rescaling, as they would have to satisfy a balancing condition in order for the corresponding heteroclinic solutions to exist (see the introduction of \cite{aftalionSourdis}).

It was shown recently in \cite{aftalionSourdis} that solutions of (\ref{eqEqGen})-(\ref{equvPos})-(\ref{eqBdryGen})   satisfying the monotonicity property (\ref{eqmonotS}) are unique up to translations; interestingly enough, it was also shown that   the monotonicity of just one of the components is enough to reach the same conclusion.

There are two singular limits associated with (\ref{eqEqGen})-(\ref{equvPos})-(\ref{eqBdryGen}): $\Lambda \to +\infty$ and $\Lambda \to 1^+$
which are called the strong and the weak separation limit, respectively. Both limits were studied formally in \cite{barankov} (see also \cite{vaninterface} and \cite{malomed1990domain} for more formal arguments in the strong and weak separation limits, respectively). In particular, it was predicted therein that the components of an energy minimizing solution satisfy $uv\to 0$ and $u^2+v^2\to 1^-$, at least pointwise, as $\Lambda \to +\infty$ and $\Lambda \to 1^+$, respectively.
The strong separation limit was studied rigorously and in great detail recently in \cite{aftalionSourdis}. The scope of the current article is to study rigorously the weak separation limit, i.e., \[\Lambda \to 1^+.\] To the best of our knowledge, the only rigorous result in this direction is contained in the recent paper \cite{goldman2015phase}, where the authors employed   $\Gamma$-convergence techniques to obtain a first order asymptotic expansion of the minimal energy.

It turns out that, in contrast to the strong separation limit, here  we can apply by now standard arguments from geometric singular perturbation theory (see \cite{kuhen} and the references therein). To this end, we first have to put system (\ref{eqEqGen}) in the appropriate slow-fast form. At this point we will rely on the intuition of the physicists in the aforementioned papers. This task will be carried out in Section \ref{secSF}. We will analyse the resulting slow-fast system using geometric singular perturbation theory in Section \ref{secGeom}. Armed with this analysis, we will prove our main result in Section \ref{secMain} which provides fine estimates for a heteroclinic solution of (\ref{eqEqGen})-(\ref{eqBdryGen}), as $\Lambda \to 1^+$, expressed in terms of suitable polar coordinates. Lastly, in Section \ref{secEnergy} we will show that this solution coincides with the unique (up to translations) minimizing heteroclic connection of (\ref{eqEqGen})-(\ref{eqBdryGen}), and provide an asymptotic expression for its energy.
 \section{The slow-fast system}\label{secSF}
We let
\begin{equation}\label{eqEpsilon}
\varepsilon=\sqrt{\Lambda-1},
\end{equation}
and consider the slow variable
\begin{equation}\label{eqxz}
x=\varepsilon z.
\end{equation}
In the rest of the paper, unless specified otherwise, we will assume that $\varepsilon>0$.
Then, system (\ref{eqEqGen}) is equivalent to
\begin{equation}\label{eqEqGenEpsilon}
	\left\{
	\begin{array}{rcl}
	\lambda^2 \varepsilon^2 u''&=&u^3-u+  v^2 u+ \varepsilon^2 v^2 u, \\
&&	\\
	  \varepsilon^2v''&=&v^3-v+  u^2 v+ \varepsilon^2 u^2 v,
	\end{array}
	\right.
	\end{equation}
where $'=d/{dx}$ (the relations (\ref{equvPos}) and (\ref{eqBdryGen}) remain the same).
Next, motivated from \cite{barankov,malomed1990domain}, we express $(u,v)$ in polar coordinates as
\begin{equation}\label{eqPolar}
u=R\cos \varphi,\ \ v=R\sin \varphi,
\end{equation}
and write (\ref{eqEqGenEpsilon})-(\ref{equvPos})-(\ref{eqBdryGen}) equivalently as
\[
\begin{split}
\varepsilon^2 \left[R''- R (\varphi')^2\right]=&(R^3-R)\left[1+\left(\frac{1}{\lambda^2}-1 \right)\cos^2 \varphi \right]\\&+\varepsilon^2R^3\left(\frac{1}{\lambda^2}+1 \right)\sin^2\varphi\cos^2\varphi,
\end{split}
\]
\[
\begin{split}
\varepsilon^2 \left(R \varphi''+2  R'\varphi'\right)=&-\left(\frac{1}{\lambda^2}-1 \right)(R^3-R) \sin \varphi \cos \varphi
\\&+\varepsilon^2 R^3\left(\sin \varphi \cos^3 \varphi-\frac{1}{\lambda^2} \cos \varphi \sin^3 \varphi\right);
\end{split}
\]
\[
R>0,\ \ 0<\varphi<\frac{\pi}{2};
\]
\[
R\to 1 \ \ \textrm{as}\ x\to \pm \infty,
\ \
\varphi \to \frac{\pi}{2}\ \textrm{as}\ x\to -\infty,\ \ \varphi \to 0\ \textrm{as}\ x\to +\infty.
\]
Subsequently, we set
\begin{equation}\label{eqR}R=1-\varepsilon^2 w, \end{equation}
and get the equivalent problem:
\[
\begin{split}
-\varepsilon^2w''- (1-\varepsilon^2 w) (\varphi')^2 =&(1-\varepsilon^2 w)(\varepsilon^2w^2-2w)\left[1+\left(\frac{1}{\lambda^2}-1 \right)\cos^2 \varphi \right]\\&+ (1-\varepsilon^2 w)^3\left(\frac{1}{\lambda^2}+1 \right)\sin^2\varphi\cos^2\varphi,
\end{split}
\]
\[
\begin{split}
(1-\varepsilon^2 w) \varphi''-2 \varepsilon^2 w'\varphi'=&\left(1-\frac{1}{\lambda^2} \right)(1-\varepsilon^2 w)(\varepsilon^2w^2-2w) \sin \varphi \cos \varphi
\\&+(1-\varepsilon^2 w)^3\left(\sin \varphi \cos^3 \varphi-\frac{1}{\lambda^2} \cos \varphi \sin^3 \varphi\right);
\end{split}
\]
\[
0<\varphi<\frac{\pi}{2};
\]
\[
w\to 0 \ \ \textrm{as}\ x\to \pm \infty,
\ \
\varphi \to \frac{\pi}{2}\ \textrm{as}\ x\to -\infty,\ \ \varphi \to 0\ \textrm{as}\ x\to +\infty.
\]
Now we can define
\begin{equation}\label{eqTransFi}
w_1=w,\ w_2=\varepsilon w_1',\ \varphi_1= \varphi,\ \varphi_2=\varphi_1',
\end{equation}
and write the problem equivalently  in the following slow-fast form, with $(w_1,w_2)$ being the fast variables and $(\varphi_1,\varphi_2)$ the slow ones:
\begin{equation}\label{eqSF}\left\{\begin{array}{rcl}
    \varepsilon w_1' & = & w_2, \\
      &   &   \\
    \varepsilon w_2' & = & - (1-\varepsilon^2 w_1) \varphi_2^2 -(1-\varepsilon^2 w_1)(\varepsilon^2w_1^2-2w_1)\left[1+\left(\frac{1}{\lambda^2}-1 \right)\cos^2 \varphi_1 \right] \\
      &   & - (1-\varepsilon^2 w_1)^3\left(\frac{1}{\lambda^2}+1 \right)\sin^2\varphi_1\cos^2\varphi_1, \\
      &   &   \\
    \varphi_1' & = & \varphi_2, \\
      &   &   \\
    \varphi_2' & = & \frac{2 \varepsilon w_2\varphi_2}{1-\varepsilon^2 w_1}+\left(1-\frac{1}{\lambda^2} \right)(\varepsilon^2w_1^2-2w_1) \sin \varphi_1 \cos \varphi_1 \\
      &   & +(1-\varepsilon^2 w_1)^2\left(\sin \varphi_1 \cos^3 \varphi_1-\frac{1}{\lambda^2} \cos \varphi_1 \sin^3 \varphi_1\right);
  \end{array}\right.
\end{equation}
\begin{equation}\label{eqPhi1pos}
0<\varphi_1<\frac{\pi}{2};
\end{equation}
\begin{equation}\label{eqSFbdryu}
\left\{\begin{array}{l}
  w_1,\ w_2\to 0 \ \ \textrm{as}\ x\to \pm \infty, \\
    \\
 \varphi_1 \to \frac{\pi}{2}\ \textrm{as}\ x\to -\infty,\ \ \varphi_1 \to 0\ \textrm{as}\ x\to +\infty,\ \ \varphi_2\to 0 \ \textrm{as}\ x\to \pm \infty.
\end{array}\right.
\end{equation}
\subsection{Analysis at the equilibria}\label{subsubLinearization}
It is easy to check that the eigenvalues of the linearization of (\ref{eqSF}) at the equilibria $(0,0,\frac{\pi}{2},0)$ and $(0,0,0,0)$ that we wish to connect are
\begin{equation}\label{eqEVs}
\pm \frac{\sqrt{2}}{\varepsilon}, \  \pm \frac{1}{\lambda}
\ \
\textrm{and}
\ \
\pm \frac{\sqrt{2}}{\lambda \varepsilon}, \  \pm 1,
\end{equation}
respectively. Moreover, as associated eigenfunctions we can choose the following:
\begin{equation}\label{eqEigen}
\left(\pm \frac{1}{\sqrt{2}},1,0,0 \right),\ \left(0,0,\pm \lambda,1 \right)\ \ \textrm{and}\ \ \left(\pm \frac{\lambda}{\sqrt{2}},1,0,0 \right),\ \left(0,0,\pm 1,1 \right),
\end{equation}
respectively.
\section{Geometric singular perturbation theoretic  analysis}\label{secGeom}
Having put the problem in the standard slow-fast form, we can now start analyzing it using geometric singular perturbation theory.
\subsection{The $\varepsilon=0$ limit slow system}\label{subsecSFslow}
The slow-fast system (\ref{eqSF}) is in the so called slow form. Switching back to the variable $z$ (recall (\ref{eqxz})) gives us the corresponding fast form. They are equivalent as long as $\varepsilon$ is positive, but they provide different information when we formally set $\varepsilon=0$. For the problem at hand, we will only need the information that comes from the slow $\varepsilon=0$ limit problem, which is the following:
\begin{equation}\label{eqSFlimit}\left\{\begin{array}{rcl}
    0 & = & w_2, \\
      &   &   \\
    0 & = & -  \varphi_2^2 +2w_1\left[1+\left(\frac{1}{\lambda^2}-1 \right)\cos^2 \varphi_1 \right]  - \left(\frac{1}{\lambda^2}+1 \right)\sin^2\varphi_1\cos^2\varphi_1, \\
      &   &   \\
    \varphi_1' & = & \varphi_2, \\
      &   &   \\
    \varphi_2' & = & -2\left(1-\frac{1}{\lambda^2} \right)w_1 \sin \varphi_1 \cos \varphi_1 +\sin \varphi_1 \cos^3 \varphi_1-\frac{1}{\lambda^2} \cos \varphi_1 \sin^3 \varphi_1.
  \end{array}\right.
\end{equation}
\subsubsection{The critical manifold $\mathcal{M}_0$} The first two equations of (\ref{eqSFlimit}) define the \emph{critical manifold}, which is
\begin{equation}\label{eqCritiqMan}
\mathcal{M}_0=\left\{w_1=\frac{\varphi_2^2 +  \left(\frac{1}{\lambda^2}+1 \right)\sin^2\varphi_1\cos^2\varphi_1}{2\left[1+\left(\frac{1}{\lambda^2}-1 \right)\cos^2 \varphi_1 \right]},\ w_2=0,\ (\varphi_1,\varphi_2)\in \mathbb{R}^2 \right\}.
\end{equation}
\subsubsection{The reduced problem}\label{subsecRedux}
The last two equations of (\ref{eqSFlimit}) define a flow on the critical manifold $\mathcal{M}_0$, which is given by the lifting on $\mathcal{M}_0$ of the trajectories of the following two-dimensional \emph{reduced system}:
\begin{equation}\label{eqRedux}\left\{\begin{array}{rcl}
    \varphi_1' & = & \varphi_2, \\
      &   &   \\
    \varphi_2' & = & -\left(1-\frac{1}{\lambda^2} \right)\left[\frac{\varphi_2^2 +  \left(\frac{1}{\lambda^2}+1 \right)\sin^2\varphi_1\cos^2\varphi_1}{1+\left(\frac{1}{\lambda^2}-1 \right)\cos^2 \varphi_1 }\right] \sin \varphi_1 \cos \varphi_1 \\ & &\\ & &+\sin \varphi_1 \cos^3 \varphi_1-\frac{1}{\lambda^2} \cos \varphi_1 \sin^3 \varphi_1.
  \end{array}\right.
\end{equation}
The form of the above system may be discouraging at first sight, but a closer look reveals that it can be written in the following simple form for $\varphi_1$:
\begin{equation}\label{eqHamReducia}
\frac{d}{dx}\left\{\left[1+\left(\frac{1}{\lambda^2}-1 \right)\cos^2 \varphi_1 \right](\varphi_1')^2 \right\}=\frac{1}{4\lambda^2}\frac{d}{dx}\left\{\sin^2(2\varphi_1) \right\}.
\end{equation}
Then, in view of the asymptotic behaviour (\ref{eqSFbdryu}), the \emph{reduced problem} becomes
\begin{equation}\label{eqRB}
\left\{\begin{array}{l}
         \varphi_1'  =-\frac{1}{2\lambda}\sin(2\varphi_1)\left[1+\left(\frac{1}{\lambda^2}-1 \right)\cos^2 \varphi_1 \right]^{-\frac{1}{2}}, \\
           \\
          \varphi_1 \to \frac{\pi}{2}\ \textrm{as}\ x\to -\infty,\ \ \varphi_1 \to 0\ \textrm{as}\ x\to +\infty.
       \end{array}\right.
\end{equation}
Clearly, the above problem admits a unique solution $\varphi_{1,0}$ such that $\varphi_{1,0}(0)=\frac{\pi}{4}$. Moreover, it holds $\varphi_{2,0}=\varphi_{1,0}'<0$. We note that this limit problem also arose in the $\Gamma$-convergence argument of \cite{goldman2015phase}. The lifting of the orbit $(\varphi_{1,0},\varphi_{2,0})$ on the critical manifold $\mathcal{M}_0$ is called \emph{singular heteroclinic orbit or connection}. We note that $(\frac{\pi}{2},0)$ and $(0,0)$ are saddle equilibria for (\ref{eqRedux}) with corresponding eigenvalues $\pm \frac{1}{\lambda}$ and $\pm 1$, respectively; the associated eigenvectors are $\left(\pm \lambda,1 \right)$ and $\left(\pm 1,1 \right)$, respectively. It is useful to compare with Subsection \ref{subsubLinearization}.
\subsection{The locally invariant manifold $\mathcal{M}_\varepsilon$}
\subsubsection{Normal hyperbolicity of $\mathcal{M}_0$} The critical manifold $\mathcal{M}_0$ corresponds to a two-dimensional manifold of equilibria for the $\varepsilon=0$ limit fast system (recall the discussion in the beginning of Subsection \ref{subsecSFslow}). The associated linearization at such an equilibrium point is
\[
\left( \begin{array}{cccc}
         0 & 1 & 0 & 0 \\
         2+2\left(\frac{1}{\lambda^2}-1 \right)\cos^2 \varphi_1 & 0 & 0 & 0 \\
         0 & 0 & 0 & 0 \\
         0 & 0 & 0 & 0
       \end{array}
\right).
\]
The eigenvalues of this matrix are $\pm \sqrt{2+2\left(\frac{1}{\lambda^2}-1 \right)\cos^2 \varphi_1}$ and zero (double). Therefore, as there are no other eigenvalues on the imaginary axis besides of  zero whose multiplicity is equal to the dimension of $\mathcal{M}_0$, we infer that the critical manifold  $\mathcal{M}_0$ is normally hyperbolic.
\subsubsection{Persistence of $\mathcal{M}_0$ for $0<\varepsilon \ll 1$}\label{subsubPersist}
Since $\mathcal{M}_0$ is normally hyperbolic and a $C^\infty$ graph over the $(\varphi_1,\varphi_2)$ plane, as a particular consequence of Fenichel's first theorem (see \cite{fenichelJDE}, \cite{jones1995geometric} or \cite[Ch. 3]{kuhen}), we deduce that, given an integer $m\geq 1$ and a compact subset $\mathcal{K}$ of the $(\varphi_1,\varphi_2)$ plane, there are functions $h_i(\varphi_1,\varphi_2,\varepsilon)\in C^m\left(\mathcal{K}\times [0,\infty) \right)$, $i=1,2$, and an $\varepsilon_0>0$ so that for $\varepsilon \in (0,\varepsilon_0)$ the graph $\mathcal{M}_\varepsilon$ over $\mathcal{K}$ described by \begin{equation}\label{eqh}w_1=\frac{\varphi_2^2 +  \left(\frac{1}{\lambda^2}+1 \right)\sin^2\varphi_1\cos^2\varphi_1}{2\left[1+\left(\frac{1}{\lambda^2}-1 \right)\cos^2 \varphi_1 \right]}+\varepsilon h_1(\varphi_1,\varphi_2,\varepsilon),\ \ w_2=\varepsilon h_2(\varphi_1,\varphi_2,\varepsilon),\end{equation}
is locally invariant under (\ref{eqSF}). In passing, we note that this property also follows by appending the equation $\dot{\varepsilon}=0$ to the equivalent fast form of (\ref{eqSF}), applying the usual center manifold theorem at each equilibrium on $\mathcal{M}_0\times \{0\}$, and then taking slices for $\varepsilon$ fixed (see \cite[Ch. 2]{berglund2006noise}). As a center-like manifold, $\mathcal{M}_\varepsilon$ is generally not unique. We choose the compact set $\mathcal{K}$ to be the closure of a smooth domain that contains the heteroclinic connection $(\varphi_{1,0}, \varphi_{2,0})$ of the reduced system (\ref{eqRedux}).  The equilibria $(0,0,\frac{\pi}{2},0)$ and $(0,0,0,0)$ of (\ref{eqSF})  lie on $\mathcal{M}_\varepsilon$, that is
\begin{equation}\label{eqEquiv0}h_i\left(\frac{\pi}{2},0,\varepsilon\right)=0,\ \ h_i\left(0,0,\varepsilon\right)=0,\ \ i=1,2,\ \ \varepsilon \in [0,\varepsilon_0).\end{equation} This is because every invariant set of (\ref{eqSF}) in a sufficiently small $\varepsilon$-independent neighborhood of $\mathcal{M}_0$ must be on $\mathcal{M}_\varepsilon$.
\subsubsection{Equivariant aspects of $\mathcal{M}_\varepsilon$}
In this subsection, we will discuss some symmetry properties of $\mathcal{M}_\varepsilon$ that are inherited from (\ref{eqSF}). We point out that these properties will only be used in order to get precise exponents in the exponential decay rates in (\ref{eqturn}). More precisely, we will just use that $\mathcal{M}_\varepsilon$ may be assumed to be tangential to $\mathcal{M}_0$ at either one of the equilibria that we wish to connect (see (\ref{eqEquiv1}) below). Therefore, depending on the reader's preference, this subsection may be skipped at first reading.

We observe that if $(w_1,w_2,\varphi_1,\varphi_2)$ solves (\ref{eqSF}), then so do \begin{equation}\label{eqSymas}(w_1,w_2,-\varphi_1,-\varphi_2)\ \ \textrm{and}\ \ (w_1,w_2,\pi-\varphi_1,-\varphi_2).\end{equation} Then, by further assuming that $\mathcal{K}$ is symmetric with respect to the lines $\varphi_1=0$, $\varphi_1=\frac{\pi}{2}$ and $\varphi_2=0$, the invariant manifold $\mathcal{M}_\varepsilon$ can be constructed so that the flow on it preserves at least one of these two properties. More precisely, we may assume that one of the following identities holds:
\begin{equation}\label{eqEquiv1}h_i\left(-\varphi_1,-\varphi_2,\varepsilon\right)=h_i\left(\varphi_1,\varphi_2,\varepsilon\right)\ \textrm{or} \ h_i\left(\pi-\varphi_1,-\varphi_2,\varepsilon\right)=h_i\left(\varphi_1,\varphi_2,\varepsilon\right),\end{equation}
for $i=1,2$ and $\varepsilon \in [0,\varepsilon_0)$. In any case, we can always assume $h_i(\cdot,\cdot,\varepsilon)$, $i=1,2$, to be even with respect to $\varphi_2$.

This follows from the way that $\mathcal{M}_\varepsilon$ is constructed (see \cite{jones1995geometric}), which we briefly recall. Firstly, one appropriately modifies the last two equations of (\ref{eqSF}) outside of $\mathcal{K}$ and constructs a unique, three-dimensional, positively invariant center-stable manifold for that modified system (note that the last relation in page 67 of the aforementioned reference should be with the opposite sign). Similarly, one constructs a unique, three-dimensional, negatively invariant, center-unstable manifold for an analogous extension of (\ref{eqSF}).
%The locally invariant manifold $\mathcal{M}_\varepsilon$ is the intersection of these two manifolds over $\mathcal{K}$.
It is easy to see that  these two modifications can be performed
while preserving one of  the symmetries in (\ref{eqSymas}).
%More precisely, in the case of the first symmetry, the cutoff function in \cite{jones1995geometric} should be chosen even in both $\mathcal{\varphi}_1$ and $\mathcal{\varphi}_2$, whereas it   should be chosen symmetric with respect to $\varphi_1=\frac{\pi}{2}$ for the second symmetry. Clearly, we cannot do both at the same time but this issue will be dealt with easily once we construct the desired heteroclinic connection.
In turn, as a consequence of their uniqueness, the corresponding center-stable and center-unstable  manifolds inherit the chosen symmetry. In particular, so does their intersection over $\mathcal{K}$, namely $\mathcal{M}_\varepsilon$. For related arguments, we refer the interested reader to \cite[Sec. 5.7]{chossat2000methods} and \cite[Ap. B]{haragus2010local}.

Let us henceforth assume that the locally invariant manifold $\mathcal{M}_\varepsilon$ enjoys the first symmetry in (\ref{eqSymas}), that is the first relation in (\ref{eqEquiv1}) holds. However, as we will see, the second relation in (\ref{eqEquiv1}) will be a-posteriori satisfied along the heteroclinic orbit on $\mathcal{M}_\varepsilon$ that we will construct in Theorem \ref{thmMan} below.
\section{The main result}\label{secMain}
We are now all set for our main result.
\begin{thm}\label{thmMan}
For each $\varepsilon>0$ sufficiently small, there is a heteroclinic orbit $(w_{1,\varepsilon},w_{2,\varepsilon},\varphi_{1,\varepsilon},\varphi_{2,\varepsilon})$ of (\ref{eqSF}) connecting the equilibria $(0,0,\frac{\pi}{2},0)$ and $(0,0,0,0)$  which lies on $\mathcal{M}_\varepsilon$. More precisely, the following estimates hold:
\begin{equation}\label{eqturn}
\begin{array}{l}
  w_{1,\varepsilon}=\frac{\varphi_{2,\varepsilon}^2 +  \left(\frac{1}{\lambda^2}+1 \right)\sin^2\varphi_{1,\varepsilon}\cos^2\varphi_{1,\varepsilon}}{2\left[1+\left(\frac{1}{\lambda^2}-1 \right)\cos^2 \varphi_{1,\varepsilon} \right]}+\mathcal{O}(\varepsilon)\min\left\{e^{\frac{2x}{\lambda}},e^{-2x} \right\}, \\
    \\
  w_{2,\varepsilon}=\mathcal{O}(\varepsilon)\min\left\{e^{\frac{2x}{\lambda}},e^{-2x} \right\}, \\
    \\
  \varphi_{i,\varepsilon}=\varphi_{i,0}+\mathcal{O}(\varepsilon)\min\left\{e^{\frac{x}{\lambda}},e^{-x} \right\},\ \ i=1,2,
\end{array}
\end{equation}
uniformly in $\mathbb{R}$, as $\varepsilon \to 0$. Moreover, it holds \begin{equation}\label{eqThmMonot}\varphi_{2,\varepsilon}<0.\end{equation}
\end{thm}
\begin{proof}
In light of the analysis in Subsection \ref{subsubLinearization}, each of the two equilibria has a two-dimensional (global) stable and unstable manifold, which is tangent at that point to the corresponding two-dimensional eigenspace in (\ref{eqEigen}). Let us call them $W_\varepsilon^s(0,0,\frac{\pi}{2},0)$, $W_\varepsilon^u(0,0,\frac{\pi}{2},0)$ and
 $W_\varepsilon^s(0,0,0,0)$, $W_\varepsilon^u(0,0,0,0)$.
The first two eigenvalues in  each relation of (\ref{eqEVs}) correspond to motion normal to $\mathcal{M_\varepsilon}$, while the latter two correspond to motion on $\mathcal{M_\varepsilon}$. The dynamical system within $\mathcal{M_\varepsilon}$ therefore has a saddle point at each of these equilibria, with one-dimensional stable and unstable manifolds given by $W_\varepsilon^s(0,0,\frac{\pi}{2},0)\cap \mathcal{M_\varepsilon}$,  $W_\varepsilon^u(0,0,\frac{\pi}{2},0)\cap \mathcal{M_\varepsilon}$ and
 $W_\varepsilon^s(0,0,0,0)\cap \mathcal{M_\varepsilon}$, $W_\varepsilon^u(0,0,0,0)\cap \mathcal{M_\varepsilon}$. Our goal is to show that $W_\varepsilon^u(0,0,\frac{\pi}{2},0)\cap \mathcal{M_\varepsilon}$ and $W_\varepsilon^s(0,0,0,0)\cap \mathcal{M_\varepsilon}$ meet. Thus, since they are one-dimensional, they have to coincide.

 We begin by deriving the equations on $\mathcal{M}_\varepsilon$. By virtue of (\ref{eqh}),  the flow of (\ref{eqSF}) on $\mathcal{M}_\varepsilon$ is determined by a smooth, for $\varepsilon \in [0,\varepsilon_0)$, $\mathcal{O}(\varepsilon)$-regular perturbation of the reduced system (\ref{eqRedux}). We will refer to this  as the \emph{$\varepsilon$-reduced system}.
  Thanks to (\ref{eqEquiv0}), the points $(\frac{\pi}{2},0)$ and $(0,0)$ are saddles for the $\varepsilon$-reduced system with associated linearized eigenvalues and eigenfunctions given by smooth $\mathcal{O}(\varepsilon)$-regular perturbations, for $\varepsilon \in [0,\varepsilon_0)$, of the corresponding ones at the end of Subsection \ref{subsecRedux}. Actually, as we have assumed the validity of the first condition in (\ref{eqEquiv1}), the corresponding linearization at $(0,0)$ is independent of $\varepsilon \in [0,\varepsilon_0)$. Our interest will be in the unstable manifold $W^u_\varepsilon(\frac{\pi}{2},0)$ of $(\frac{\pi}{2},0)$ and in the stable manifold $W^s_\varepsilon (0,0)$  of $(0,0)$. In fact, these are the projections to the $(\varphi_1,\varphi_2)$ plane of $W_\varepsilon^u(0,0,\frac{\pi}{2},0)\cap \mathcal{M_\varepsilon}$ and $W_\varepsilon^s(0,0,0,0)\cap \mathcal{M_\varepsilon}$, respectively.

     The manifolds $W^u_\varepsilon(\frac{\pi}{2},0)$ and  $W^s_\varepsilon (0,0)$ depend smoothly on $\varepsilon \in [0,\varepsilon_0)$ (see for instance \cite[Ch. 9]{teschl2012ordinary}). From now on, with this notation, we will only refer to the parts of these invariant manifolds that shadow the heteroclinic orbit $(\varphi_{1,0}, \varphi_{2,0})$. Then, $W^u_\varepsilon(\frac{\pi}{2},0)$ and  $W^s_\varepsilon (0,0)$ intersect the line $\phi_1=\frac{\pi}{4}$ at the points $(\frac{\pi}{4},\phi_{2,\varepsilon}^-)$
  and $(\frac{\pi}{4},\phi_{2,\varepsilon}^+)$, respectively, such that
  \begin{equation}\label{eqContradict0}
\phi_{2,\varepsilon}^\pm-\varphi_{2,0}(0)=\mathcal{O}(\varepsilon)\ \ \textrm{as}\ \varepsilon \to 0,
  \end{equation}
  (recall Subsection \ref{subsecRedux}).
   Let $\left(w_{1,\varepsilon}^-, w_{2,\varepsilon}^-,\frac{\pi}{4},\phi_{2,\varepsilon}^-\right)$ and $\left(w_{1,\varepsilon}^+, w_{2,\varepsilon}^+,\frac{\pi}{4},\phi_{2,\varepsilon}^+\right)$, respectively,  be their lifting to $\mathcal{M}_\varepsilon$ for $\varepsilon \in [0,\varepsilon_0)$. The values $w_{i,\varepsilon}^\pm$, $i=1,2$, depend smoothly on $\varepsilon \in [0,\varepsilon_0)$; in particular, it holds
    \begin{equation}\label{eqContra2}
    w_{i,\varepsilon}^\pm-w_{i,0}=\mathcal{O}(\varepsilon),\ i=1,2,\  \textrm{as}\ \varepsilon \to 0,
    \end{equation}
    where $\left(w_{1,0},w_{2,0} \right)$ is the image of $\left(\frac{\pi}{4}, \varphi_{2,0}(0) \right)$ on the graph of $\mathcal{M}_0$.
    We will show that \begin{equation}\label{eqMeet}w_{i,\varepsilon}^-=w_{i,\varepsilon}^+,\ i=1,2,\ \textrm{and}\ \ \phi_{2,\varepsilon}^-=\phi_{2,\varepsilon}^+ ,\end{equation}
  provided that $\varepsilon>0$ is sufficiently small.

  Notice that we want to determine uniquely three variables, although (\ref{eqh}) furnishes only two equations. The third equation will be provided by the hamiltonian identity (\ref{eqHam}) (see also \cite{alikakos2007singular} for a related argument in a simpler problem).
Taking into account (\ref{eqEpsilon}), (\ref{eqxz}), (\ref{eqPolar}), (\ref{eqR}), and dividing by $\varepsilon^2/2$, we find that the identity (\ref{eqHam}) becomes
 \begin{equation}\label{eqRHS}\begin{split}0
 =& \lambda^2\left[\varepsilon^2 w_2^2 \cos ^2 \varphi_1+(1-\varepsilon^2 w_1)^2\varphi_2^2 \sin^2\varphi_1+\varepsilon w_2(1-\varepsilon^2 w_1)\sin 2 \varphi_1 \right]
  \\
 & +\varepsilon^2 w_2^2 \sin ^2 \varphi_1+(1-\varepsilon^2 w_1)^2\varphi_2^2 \cos^2\varphi_1-\varepsilon w_2 (1-\varepsilon^2 w_1) \sin2 \varphi_1 \\
 & -\frac{\varepsilon^2}{2}(2w_1-\varepsilon^2 w_1^2)^2-\frac{1}{4}(1-\varepsilon^2 w_1)^4 \sin^22 \varphi_1,
\end{split}
\end{equation}
which is valid along trajectories of (\ref{eqSF}) on either one of $W_\varepsilon^{s/u}\left(0,0,\frac{\pi}{2},0\right)$ or $W_\varepsilon^{s/u}\left(0,0,0,0\right)$, for $\varepsilon>0$.
Moreover, it will be important in the sequel to observe that, thanks to (\ref{eqHamReducia}), the above identity continues to hold for $\varepsilon=0$, i.e., along $(\varphi_{1,0},\varphi_{2,0})$.

We consider the smooth map $F: \mathbb{R}^2 \times \mathcal{K}\times [0,\infty)\to \mathbb{R}^3$ defined by
\[F\left(\begin{array}{c}
     w_1 \\
     w_2  \\
     \varphi_1 \\
     \varphi_2 \\
     \varepsilon
   \end{array}\right)
=\left(\begin{array}{c}
   w_1-\frac{\varphi_2^2 +  \left(\frac{1}{\lambda^2}+1 \right)\sin^2\varphi_1\cos^2\varphi_1}{2\left[1+\left(\frac{1}{\lambda^2}-1 \right)\cos^2 \varphi_1 \right]}-\varepsilon h_1(\varphi_1,\varphi_2,\varepsilon) \\
      \\
    w_2-\varepsilon h_2(\varphi_1,\varphi_2,\varepsilon) \\
      \\
    H(w_1,w_2,\varphi_1,\varphi_2, \varepsilon)
  \end{array}\right),
\]
where $H$ is the function defined by the righthand side of (\ref{eqRHS}).
We observe that
\begin{equation}\label{eqContraFinale}F\left(w_{1,\varepsilon}^\pm,w_{2,\varepsilon}^\pm,\frac{\pi}{4},\phi_{2,\varepsilon}^\pm,\varepsilon \right)=(0,0,0),\ \ \varepsilon \in (0,\varepsilon_0).\end{equation}
Furthermore, it holds \begin{equation}\label{eqContra3}F\left(w_{1,0},w_{2,0},\frac{\pi}{4},\phi_{2,0}(0),0 \right)=(0,0,0).\end{equation}
Moreover, it follows readily that
\begin{equation}\label{eqital}
\partial_{w_1,w_2,\varphi_2}F\left(\begin{array}{c}
     w_1 \\
     w_2  \\
     \varphi_1 \\
     \varphi_2 \\
     0
   \end{array}\right)=\left(\begin{array}{ccc}
                              1 & 0 & -\frac{\varphi_2}{1+\left(\frac{1}{\lambda^2}-1 \right)\cos^2 \varphi_1} \\
                              0 & 1 & 0 \\
                              0 & 0 & \lambda^2 \varphi_2 \sin^2\varphi_1 +\varphi_2 \cos^2\varphi_1
                            \end{array}
    \right).
\end{equation}
In particular,  this matrix is invertible at the point $\left(w_{1,0},w_{2,0},\frac{\pi}{4},\varphi_{2,0}(0),0 \right)$.
Thus, recalling (\ref{eqContra3}), we deduce by the implicit function theorem that there exists $\delta>0$ such that, for $\varphi_1 \in \left(\frac{\pi}{4}-\delta,\frac{\pi}{4}+\delta \right)$ and $\varepsilon \in [0, \delta)$, the equation
\[
F\left(w_1,w_2,\varphi_1,\varphi_2,\varepsilon \right)=(0,0,0)
\]
has a unique solution $(w_1,w_2,\varphi_2)$ such that $|w_i-w_{i,0}|<\delta$, $i=1,2$, and $\left|\varphi_2-\varphi_{2,0}(0) \right|<\delta$. Hence, applying this property for $\varphi_1=\frac{\pi}{4}$, we infer from (\ref{eqContradict0}), (\ref{eqContra2}) and (\ref{eqContraFinale}) that the desired relation
 (\ref{eqMeet}) is true, provided that $\varepsilon>0$ is sufficiently small.

Let $\left(w_{1,\varepsilon},w_{2,\varepsilon},\varphi_{1,\varepsilon},\varphi_{2,\varepsilon} \right)$
denote the heteroclinic connection of (\ref{eqSF}), (\ref{eqSFbdryu}) on $\mathcal{M}_\varepsilon$ which passes through the point
$\left(w_{1,\varepsilon}^+,w_{2,\varepsilon}^+,\frac{\pi}{4},\phi_{2,\varepsilon}^+ \right)$ at $x=0$. We will first establish the validity of properties (\ref{eqPhi1pos}) and (\ref{eqThmMonot}). For this purpose, we recall that the trajectory curve of $(\varphi_{1,\varepsilon},\varphi_{2,\varepsilon})$ on the $(\varphi_{1},\varphi_{2})$ phase plane is  given by $W_\varepsilon^u\left(\frac{\pi}{2},0\right)\cap W^s_\varepsilon(0,0)$, and varies smoothly for $\varepsilon \geq 0$ small.  The asserted properties now follow at once from the fact that the limiting curve $W_0^u\left(\frac{\pi}{2},0\right)\cap W^s_0(0,0)$ is
contained in the half-strip $\mathcal{S}=\left\{ 0\leq \varphi_1 \leq \frac{\pi}{2},\ \ \varphi_2\leq 0 \right\}$, and touches the boundary of $\mathcal{S}$ only at $\left(0,0 \right)$ and $\left(\frac{\pi}{2},0 \right)$ in a non-tangential manner
(keep in mind the linearized analysis from the end of Subsection \ref{subsecRedux}).

We next turn our attention to the last relation in (\ref{eqturn}). We will first show it for $x\geq 0$. To this end, we will need the preliminary estimates
\begin{equation}\label{eqPrelim}
\varphi_{i,\varepsilon}(x) =(-1)^{i-1}a_+ \left(1+o(1)\right) e^{-x},\ \ i=1,2, \ \ \textrm{as}\ \  x\to +\infty,
\end{equation}
where the constant $a_+> 0$ is independent of small $\varepsilon >0$, and these limits hold uniformly with respect to $\varepsilon$. The above relation follows directly from the refined version of the stable manifold theorem in \cite[Thm. 4.3]{coddington1955theory}; recall that the linearization of the $\varepsilon$-reduced system at $(0,0)$ has eigenvalues $\pm 1$ for $\varepsilon \geq 0$ small. The latter property about the linearized problem  implies that the pair $\Psi_\varepsilon=(\psi_{1,\varepsilon},\psi_{2,\varepsilon})$, where
\[
\psi_{i,\varepsilon}=\frac{\varphi_{i,\varepsilon}-\varphi_{i,0}}{\varepsilon},\ \ i=1,2,
\]
satisfies the following:
\[\left\{\begin{array}{l}
    \Psi_\varepsilon'=A\Psi_\varepsilon+\mathcal{O}\left(\varepsilon|\Psi_\varepsilon|^2\right)+\mathcal{O}\left(\varphi_{1,\varepsilon}^2+\varphi_{2,\varepsilon}^2 \right),\ \ x\geq 0; \\
      \\
    \Psi_\varepsilon(0)=\mathcal{O}(1),\ \ \Psi_\varepsilon(\infty)=0,
  \end{array}\right.
\]
with the obvious notation, uniformly as $\varepsilon \to 0$, where $A$ is the aforementioned linearized matrix (recall also (\ref{eqContradict0})).
Then,  by using (\ref{eqPrelim}) to estimate the last term in the righthand side and by working as in the previously mentioned stable manifold theorem in \cite{coddington1955theory}, we obtain that
\[
\left|\Psi_\varepsilon(x)\right|\leq Ce^{-x},\ \ x\geq 0,
\]
for some constant $C>0$ independent of small $\varepsilon>0$, which implies the validity of the last relation of (\ref{eqturn}) for $x\geq 0$. In turn, the corresponding estimates in the first two relations of (\ref{eqturn}) follow at once via the second identity in (\ref{eqEquiv0}) and the first one in (\ref{eqEquiv1}).

 The sole obstruction in showing the corresponding estimates for $x\leq 0$ is that the linearization of the $\varepsilon$-reduced system at $\left(\frac{\pi}{2},0 \right)$ is not independent of $\varepsilon$ (recall that we could only choose one of the symmetries in (\ref{eqSymas})). Nevertheless, this can be surpassed easily by noting that the constructed heteroclinic connection of (\ref{eqSF}) on $\mathcal{M}_\varepsilon$ should also be on an analogous invariant manifold $\tilde{\mathcal{M}}_\varepsilon$ which enjoys the second symmetry in (\ref{eqSymas}) (recall the concluding remark in Subsection \ref{subsubPersist}), provided that $\varepsilon>0$ is sufficiently small. Then, the arguments for $x\leq 0$ go through as before. In passing, we note that the graphs of $\mathcal{M}_\varepsilon$ and $\tilde{\mathcal{M}}_\varepsilon$ over $\mathcal{K}$ have the same expansion in powers of $\varepsilon$ up to any order (see \cite[Ch. 3]{kuhen} for more details).

 The proof of the theorem is complete.
\end{proof}
\begin{rem}
We suspect that the calculation in (\ref{eqital}) provides the required nondegeneracy condition in
\cite[Sec. 5]{mackay2004slow}
which allows to choose $\mathcal{M}_\varepsilon$ so that the corresponding $\varepsilon$-reduced system is hamiltonian (in $p=\cos \phi_1$, $q=\sin \phi_1$).
\end{rem}

\begin{rem} From the invariance of $\mathcal{M}_\varepsilon$ and the equation $w_2=\varepsilon w_1'$, via the second equation of (\ref{eqh}), we obtain  that
\[\begin{array}{rcl}
    \frac{w_{2,\varepsilon}}{\varepsilon} & = & -2\left(1-\frac{1}{\lambda^2} \right)\varphi_{2,\varepsilon}\sin \varphi_{1,\varepsilon} \cos \varphi_{1,\varepsilon}\frac{\varphi_2^2 +  \left(\frac{1}{\lambda^2}+1 \right)\sin^2\varphi_{1,\varepsilon}\cos^2\varphi_{1,\varepsilon}}{\left[1+\left(\frac{1}{\lambda^2}-1 \right)\cos^2 \varphi_{1,\varepsilon} \right]^2} \\
      &   & +\frac{\varphi_{2,\varepsilon}}{\left[1+\left(\frac{1}{\lambda^2}-1 \right)\cos^2 \varphi_{1,\varepsilon} \right]}\left(\sin \varphi_1 \cos^3 \varphi_1-\frac{1}{\lambda^2} \cos \varphi_1 \sin^3 \varphi_1\right) \\
      &   & +\frac{1}{2}\left(1+\frac{1}{\lambda^2} \right)\varphi_{2,\varepsilon}\frac{\sin 2\varphi_{1,\varepsilon}-4\cos\varphi_{1,\varepsilon}\sin^3\varphi_{1,\varepsilon}}{1+\left(\frac{1}{\lambda^2}-1 \right)\cos^2 \varphi_{1,\varepsilon} }  \\
  %    &   & +\left(\frac{1}{\lambda^2}-1 \right)\varphi_{2,\varepsilon}\left(\sin 2\varphi_{1,\varepsilon}\right)\frac{\varphi_{2,\varepsilon}^2 +  \left(\frac{1}{\lambda^2}+1 \right)\sin^2\varphi_{1,\varepsilon}\cos^2\varphi_{1,\varepsilon}}{2\left[1+\left(\frac{1}{\lambda^2}-1 \right)\cos^2 \varphi_{1,\varepsilon} \right]^2} \\
      & & +\mathcal{O(\varepsilon)}\min\left\{e^{\frac{2x}{\lambda}},e^{-2x} \right\},
  \end{array}
\]
uniformly in $\mathbb{R}$ as $\varepsilon \to 0$. Analogously, we can refine the $w_1$ component of the constructed heteroclinic. Then, plugging these refinements in the $\varepsilon$-reduced system, we can refine the $\varphi_1,\varphi_2$ components too (by the solution of a linear inhomogeneous problem), and so on. We note, however, that formally the correct spatial decay in the above relation should be $\min\left\{e^{\frac{3x}{\lambda}},e^{-3x} \right\}$. This observation points in the direction that   $\mathcal{M}_\varepsilon$ should  be close beyond all orders of $\varepsilon$ to $\mathcal{M}_0$ at the two equilibria (recall the proofs of the corresponding decay estimates in (\ref{eqturn}) and the concluding remark in the proof of Theorem \ref{thmMan}).
\end{rem}
\section{Further properties of the constructed heteroclinic connection}\label{secEnergy}
\subsection{Variational characterization}
In view of (\ref{eqThmMonot})
and the comments leading to (\ref{eqARCT}),  we expect that the corresponding solution to (\ref{eqEqGen})-(\ref{eqBdryGen}), provided by Theorem \ref{thmMan} via the transformations (\ref{eqEpsilon}), (\ref{eqxz}), (\ref{eqPolar}), (\ref{eqR}) and (\ref{eqTransFi}),  minimizes the associated energy.
By the uniqueness result of \cite{aftalionSourdis} that we mentioned in the introduction, to verify this, it suffices to show that one of its components satisfies the corresponding monotonicity property in (\ref{eqmonotS}).
For this purpose,   we note that
\[
u'=-\varepsilon w_2 \cos\varphi_1- (1-\varepsilon^2 w_1)\varphi_2\sin\varphi_1.
\]
Hence, by virtue of (\ref{eqturn}) and (\ref{eqThmMonot}), given any fixed interval $I$, it holds $u'>0$ in $I$ for sufficiently small $\varepsilon>0$.
 We infer that $u'>0$ outside of $I$ by means of (\ref{eqPrelim}) (and the analogous relation for $x\leq 0$).
 Alternatively, similarly to \cite{aftalionSourdis}, we just have to fix a sufficiently large $I$ so that we can apply the maximum principle componentwise in the linear elliptic system for $u',v'$ in $\mathbb{R}\setminus I$ (note that such an interval can be chosen to be independent of $\varepsilon$).

\subsection{Energy expansion}
By exploiting the above observation and making mild use of the estimates in Theorem \ref{thmMan}, we are in position to give an asymptotic expression for the minimal energy of the heteroclinic connection problem (\ref{eqEqGen})-(\ref{eqBdryGen}) as $\Lambda \to 1^+$. The limiting value of the minimal energy, appropriately renormalized (so that it does not converge to zero), was identified rigorously very recently in \cite{goldman2015phase}, using the variational technique of $\Gamma$-convergence. We recover their result but also provide a rate of convergence to this minimal value.

\begin{pro}\label{proEnergy}
Let
\[
\sigma_\Lambda=\inf_\mathcal{X} E_\Lambda(u,v),
\]
where
\[
E_\Lambda(u,v)=\int_{-\infty}^{\infty}\left[\lambda^2\frac{(\dot{u})^2}{2}+\frac{(\dot{v})^2}{2}+\frac{(1-u^2-v^2)^2}{4}+\frac{\Lambda-1}{2}u^2v^2 \right]dz
\]
and
\[
\mathcal{X}=\left\{(u, v)\in W^{1,2}_{loc}(\mathbb{R})\times W^{1,2}_{loc}(\mathbb{R}) \ \textrm{satisfying}\ (\ref{eqBdryGen}) \right\}.
\]
It holds
\[
\sigma_\Lambda= \frac{1}{3} \frac{1-\lambda^3}{1-\lambda^2} (\Lambda-1)^{\frac{1}{2}} +\mathcal{O}\left(\Lambda-1 \right)\ \ \textrm{as}\ \Lambda \to 1^+,
\]
where $\varphi_{1,0}$ is the prescribed solution of (\ref{eqRB}) (with the obvious meaning for $\lambda=1$).
\end{pro}
\begin{proof}
It follows from (\ref{eqRHS}), paying attention to the  comment leading to it, that
\[
\sigma_\Lambda= \frac{1}{4} \left(\int_{-\infty}^{\infty}\sin^2\left(2\varphi_{1,0} \right)dx\right) (\Lambda-1)^{\frac{1}{2}} +\mathcal{O}\left(\Lambda-1 \right)\ \ \textrm{as}\ \Lambda \to 1^+.
\]
It therefore remains to compute the above integral. Using (\ref{eqHamReducia}), we find that
\[\begin{split}
\int_{-\infty}^{\infty}\sin^2\left(2\varphi_{1,0}\right) dx=& -2\lambda\int_{-\infty}^{\infty}\sin\left(2\varphi_{1,0}\right)\left[1+\left(\frac{1}{\lambda^2}-1\right)\cos^2\varphi_{1,0} \right]^\frac{1}{2}\varphi_{1,0}'dx \\
=&2\lambda \int_{0}^{1}\left[1+\left(\frac{1}{\lambda^2}-1\right)t \right]^\frac{1}{2}dt \\
=& \frac{4}{3}\frac{1-\lambda^3}{1-\lambda^2},
 \end{split}
\]
which implies the assertion of the proposition.
\end{proof}
\section*{Acknowledgments} I would like to thank Prof. Aftalion for bringing this problem to my attention and for useful discussions. Moreover, I would like to thank Prof. Scheel for bringing the paper \cite{van2000domain} to my attention, from the references therein I also found out about \cite{malomed1990domain} and \cite{van1997domain}. This project has received funding from the European Union's Horizon 2020 research and innovation programme under the Marie Sk\l{}odowska-Curie grant agreement No 609402-2020 researchers: Train to Move (T2M).

\bibliographystyle{plain}
\bibliography{biblioaacs}
\end{document}